\documentclass[a4paper, leqno]{amsart}

\usepackage{lmodern}

\usepackage[T1]{fontenc}

\usepackage{amssymb}
\usepackage[abbrev]{amsrefs}
\usepackage{microtype}
\usepackage{enumerate}

\usepackage[vmargin=3cm]{geometry}

\newtheorem*{theoremA}{Theorem A}
\newtheorem*{theoremB}{Theorem B}
\newtheorem*{theoremC}{Theorem C}
\newtheorem*{theoremAp}{Theorem A\cprime}
\newtheorem*{theoremBp}{Theorem B\cprime}
\newtheorem*{theoremCp}{Theorem C\cprime}

\newtheorem{proposition}{Proposition}[section]

\theoremstyle{definition}

\numberwithin{equation}{section}

\newcommand{\dif}{\,\mathrm{d}}

\newcommand{\R}{{\mathbb R}}

\newcommand{\Rem}{\mathcal{R}}
\newcommand{\A}{\mathcal{A}}

\newcommand{\Norm}[1]{\lVert #1 \rVert}
\newcommand{\abs}[1]{\lvert #1 \rvert}
\newcommand{\Bigabs}[1]{\Bigl\lvert #1 \Bigr\rvert}
\newcommand{\bigabs}[1]{\bigl\lvert #1 \bigr\rvert}
\newcommand{\biggabs}[1]{\biggl\lvert #1 \biggr\rvert}

\newcommand{\loc}{\mathrm{loc}}

\DeclareMathOperator{\Lin}{Lin}
\DeclareMathOperator{\supp}{supp}

\allowdisplaybreaks

\title[Nonlocal Hardy type inequality]{Nonlocal Hardy type inequalities
with optimal constants and remainder terms}
\author{Vitaly Moroz}
\address{Swansea University\\ Department of Mathematics\\ Singleton Park\\
Swansea\\ SA2~8PP\\ Wales, United Kingdom}	
\email{V.Moroz@swansea.ac.uk}

\author{Jean Van Schaftingen}
\address{Universit\'e Catholique de Louvain\\ Institut de Recherche en Math\'ematique et Physique\\
Chemin du Cyclotron 2 bte L7.01.01\\ 1348 Louvain-la-Neuve \\ Belgium}
\email{Jean.VanSchaftingen@uclouvain.be}


\begin{document}

\begin{abstract}
Using a groundstate transformation, we give a new proof of the optimal Stein-Weiss inequality of Herbst
\[
\int_{\R^N} \int_{\R^N} \frac{\varphi (x)}{\abs{x}^\frac{\alpha}{2}}
I_\alpha (x - y) \frac{\varphi (y)}{\abs{y}^\frac{\alpha}{2}}\dif x \dif y
\le
\mathcal{C}_{N,\alpha, 0}\int_{\R^N} \abs{\varphi}^2,
\]
and of its combinations with the Hardy inequality by Beckner
\[
 \int_{\R^N} \int_{\R^N} \frac{\varphi (x)}{\abs{x}^\frac{\alpha + s}{2}}
I_\alpha (x - y) \frac{\varphi (y)}{\abs{y}^\frac{\alpha + s}{2}}\dif x \dif y
\le \mathcal{C}_{N, \alpha, 1}
\int_{\R^N} \abs{\nabla \varphi}^2,
\]
and with the fractional Hardy inequality
\[
\int_{\R^N} \int_{\R^N} \frac{\varphi (x)}{\abs{x}^\frac{\alpha + s}{2}}
I_\alpha (x - y) \frac{\varphi (y)}{\abs{y}^\frac{\alpha + s}{2}}\dif x \dif y
\le \mathcal{C}_{N, \alpha, s} \mathcal{D}_{N, s} \int_{\R^N} \int_{\R^N}
\frac{\bigabs{\varphi (x) - \varphi (y) }^2}{\abs{x-y}^{N+s}}\dif x \dif y
\]
where \(I_\alpha\) is the Riesz potential, \(0 < \alpha < N\) and \(0 < s < \min(N, 2)\).
We also prove the optimality of the constants.
The method is flexible and yields a sharp expression for the remainder terms in these inequalities.
\end{abstract}


\date{\today}

\maketitle

\section{Introduction}

E.\thinspace Stein and G.\thinspace Weiss \cite{StWe1958} have proved that for every \(N \ge 1\) and \(\alpha \in (0, N)\)
there exists a constant \(C >0\) such that for every \(\varphi \in L^2 (\R^N)\),
\begin{equation}\label{in-SW}
 \int_{\R^N} \int_{\R^N} \frac{\varphi (x)}{\abs{x}^\frac{\alpha}{2}}
I_\alpha (x - y) \frac{\varphi (y)}{\abs{y}^\frac{\alpha}{2}}
\dif y\dif x
\le C \int_{\R^N} \abs{\varphi}^2,
\end{equation}
where \(I_\alpha\) is the Riesz potential defined for \(\alpha \in (0, N)\) and \(x \in \R^N \setminus \{0\}\) by
\[
I_\alpha (x) = \frac{\mathcal{A}_\alpha}{\abs{x}^{N-\alpha}},
\]
and the constant \(\A_\sigma\) is
\[
  \A_\alpha: = \frac{\Gamma \bigl(\tfrac{N - \alpha}{2}\bigr)}{2^{\alpha}\pi^{N/2} \Gamma \bigl(\tfrac{\alpha}{2}\bigr)},
\]
so that the Riesz potentials satisfy the semigroup property \(I_\alpha\ast I_\beta = I_{\alpha+\beta}\) for \(0<\alpha,\beta<N\) and \(\alpha+\beta<N\),
see \cite{Ri1949}*{p.\thinspace 19} or \cite{Landkof}*{chapter 1.1}.

The optimal constant in Stein--Weiss inequality \eqref{in-SW} was computed by I.\thinspace Herbst \cite{He1977}, who proved the following

\begin{theoremA}[I.\thinspace Herbst, 1977 \cite{He1977}*{theorem 2.5}]
For every \(N \ge 1\) and \(\alpha \in (0, N)\) it holds
\begin{equation*}
\mathcal{C}_{N, \alpha, 0}:=
\mathop{\sup_{\varphi\in L^2(\R^N)}}_{\Norm{\varphi}_{L^2}\le 1}\,\int_{\R^N} \int_{\R^N} \frac{\varphi (x)}{\abs{x}^\frac{\alpha}{2}}\,
I_\alpha (x - y)\, \frac{\varphi (y)}{\abs{y}^\frac{\alpha}{2}}\dif x \dif y
= \frac{1}{2^\alpha} \Bigl(\frac{\Gamma (\frac{N - \alpha}{4})}{\Gamma (\frac{N + \alpha}{4})} \Bigr)^2.
\end{equation*}
\end{theoremA}

Herbst's proof consists in writing the associated linear operator on \(L^2 (\R^N)\) as a convolution for the dilation group of simpler operators.
A close inspection of his proof suggests that the equality is not achieved and that almost extremizers of the inequality should be similar to \(x \mapsto \abs{x}^{-N/2}\). The proof was rediscovered independently under the name of Pitt's inequality by W.\thinspace Beckner \cite{Be1995}, who also obtained in \cite{Be2008} for \(N \ge 3\) the optimal constant for the combination of the Stein--Weiss inequality with the classical Hardy inequality,
\begin{equation*}
\int_{\R^N} \int_{\R^N} \frac{\varphi (x)}{\abs{x}^\frac{\alpha + 2}{2}}
I_\alpha (x - y) \frac{\varphi (y)}{\abs{y}^\frac{\alpha + 2}{2}}\dif x \dif y
\le \mathcal{C}_{N, \alpha, 1}\int_{\R^N} \abs{\nabla \varphi}^2,
\end{equation*}
which holds for every \(\varphi\in \dot{H}^1(\R^N)\).
Here \(\dot{H}^1 (\R^N)\) is the homogeneous Sobolev space,
obtained by completion of \(C^\infty_c (\R^N)\) with respect to the norm \(\Norm{\cdot}_{\dot{H}^1}\) defined by
\(\Norm{\varphi}_{\dot{H}^1}^2=\int_{\R^N}|\nabla\varphi|^2 \dif x\).

\begin{theoremB}[W.\thinspace Beckner, 2008 \cite{Be2008}*{theorem 4}]
For every \(N \ge 3\) and \(\alpha \in (0, N)\), it holds
\begin{equation*}
\mathcal{C}_{N, \alpha, 1}:=\mathop{\sup_{\varphi\in H^1(\R^N)}}_{\Norm{\nabla\varphi}_{L^2}\le 1}\,\int_{\R^N} \int_{\R^N} \frac{\varphi (x)}{\abs{x}^\frac{\alpha+2}{2}}
I_\alpha (x - y) \frac{\varphi (y)}{\abs{y}^\frac{\alpha+2}{2}}\dif x \dif y
=\frac{1}{2^{\alpha-2}} \Bigl(\frac{\Gamma (\frac{N - \alpha}{4})}{(N-2)\Gamma (\frac{N + \alpha}{4})} \Bigr)^2.
\end{equation*}
\end{theoremB}

In the present note, we give a simple new proof of theorems A and B. Our proof is based on groundstate transformations in the spirit of Agmon--Allegretto--Piepenbrink \citelist{\cite{Allegretto-74}\cite{Piepenbrink-74}\cite{Ag1983}}, which allow to derive sharp remainder representations in both inequalities. Herbst's inequality (theorem A) can be deduced from the following identity

\begin{theoremAp}
If \(N \ge 1\) and \(\alpha \in (0, N)\), then for every \(\varphi \in L^2 (\R^N)\) it holds
\begin{multline*}
\mathcal{C}_{N, \alpha, 0}
\int_{\R^N} \abs{\varphi}^2
= \int_{\R^N} \int_{\R^N} \frac{\varphi (x)}{\abs{x}^\frac{\alpha}{2}}
I_\alpha (x - y) \frac{\varphi (y)}{\abs{y}^\frac{\alpha}{2}}\dif x \dif y \\
+ \frac{1}{2} \int_{\R^N} \int_{\R^N}
\frac{I_\alpha (x - y)}{\abs{x}^\frac{N + \alpha}{2} \abs{y}^\frac{N + \alpha}{2}}
\bigabs{\varphi (x)\abs{x}^\frac{N}{2} - \varphi (y) \abs{y}^\frac{N}{2}}^2\dif x \dif y.
\end{multline*}
\end{theoremAp}

Beckner's inequality (theorem~B) is a consequence of its quantitative version:

\begin{theoremBp}
If \(N \ge 3\) and \(\alpha \in (0, N)\), then for every \(\varphi \in \dot{H}^1(\R^N)\) it holds
\begin{multline*}
\mathcal{C}_{N, \alpha, 1} \biggl(\int_{\R^N} \abs{\nabla \varphi}^2 - \int_{\R^N} \frac{\bigabs{\nabla \bigl(\abs{x}^\frac{N - 2}{2} \varphi (x)\bigr)}^2}{\abs{x}^{N - 2}} \dif x  \biggr)\\*
\shoveleft{\qquad\qquad=  \int_{\R^N} \int_{\R^N} \frac{\varphi (x)}{\abs{x}^\frac{\alpha + 2}{2}}
I_\alpha (x - y) \frac{\varphi (y)}{\abs{y}^\frac{\alpha + 2}{2}}\dif x \dif y} \\
+ \frac{1}{2} \int_{\R^N} \int_{\R^N}
\frac{I_\alpha (x - y)}{\abs{x}^\frac{N + \alpha}{2} \abs{y}^\frac{N + \alpha}{2}}
\bigabs{\varphi (x)\abs{x}^\frac{N - 2}{2} - \varphi (y) \abs{y}^\frac{N - 2}{2}}^2\dif x \dif y.
\end{multline*}
\end{theoremBp}

In the limiting case \(\alpha = 0\), \(I_\alpha\) is Dirac's delta and we recover
the Agmon--Allegretto--Piepenbrink groundstate representation
\cites{Allegretto-74,Piepenbrink-74,Ag1983} for the classical local Hardy's inequality,
\[
  \int_{\R^N} \abs{\nabla\varphi}^2 \dif x =\Big(\frac{N - 2}{2}\Big)^2\int_{\R^N} \frac{\abs{\varphi(x)}^2}{\abs{x}^2}\dif x + \int_{\R^N} \frac{\bigabs{\nabla \bigl(\abs{x}^\frac{N - 2}{2} \varphi (x)\bigr)}^2}{\abs{x}^{N - 2}} \dif x .
\]

Our proof of theorems~A\cprime{} and B\cprime{} combines previously known groundstate representations with a novel version
of a groundstate representations which is designed to handle the nonlocal term on the right-hand side.

Our method is flexible enough to establish the optimal constants and sharp remainder representations in a family of nonlocal Hardy type inequalities,
which includes theorems~A\cprime{} and B\cprime{} as the limit cases. We prove

\begin{theoremC}
Let \(N \ge 1\), \(\alpha \in (0, N)\), \(s \in (0, 2)\) and \(s<N\). Then
\begin{equation*}
\mathcal{C}_{N, \alpha, s}:=\mathop{\sup_{\varphi\in \dot{H}^{s/2}(\R^N)}}_{\Norm{\varphi}_{{\dot{H}^{s/2}}}\le 1}\,\int_{\R^N} \int_{\R^N} \frac{\varphi (x)}{\abs{x}^\frac{\alpha+s}{2}}
I_\alpha (x - y) \frac{\varphi (y)}{\abs{y}^\frac{\alpha+s}{2}}\dif x \dif y  \\[-3ex]
= \frac{1}{2^{\alpha+s}}\Bigl(\frac{\Gamma\big(\frac{N-s}{4}\big)\Gamma\big(\frac{N-\alpha}{4}\big)}
{\Gamma\big(\frac{N+s}{4}\big)\Gamma\big(\frac{N+\alpha}{4}\big)}\Bigr)^2.
\end{equation*}
\end{theoremC}

Here \(\dot{H}^{s}(\R^N)\) denotes the homogeneous Sobolev space obtained by completion of
\(C^\infty_c (\R^N)\) with respect to the norm \(\Norm{\cdot}_{\dot{H}^\frac{s}{2}}\) defined by
\[
 \Norm{\varphi}_{\dot{H}^\frac{s}{2}}^2 =  \mathcal{D}_{N, s}\int_{\R^N}\int_{\R^N}\frac{\bigabs{\varphi(x)-\varphi(y)}^2}{\abs{x-y}^{N+s}}\dif x  \dif y,
\]
where
\[
  \mathcal{D}_{N, s} = \frac{\Gamma\bigl(\tfrac{N + s}{2}\bigr) s}{2^{2 - s} \pi^{N/2} \Gamma\bigl(1 - \frac{s}{2} \bigr)}.
\]
The constant \(\mathcal{D}_{N, s}\) ensures that \(\lim_{s \to 0} \Norm{\varphi}_{\dot{H}^{s/2}} = \Norm{\varphi}_{L^2}\) and \(\lim_{s \to 2} \Norm{\varphi}_{\dot{H}^{s/2}} = \Norm{\nabla \varphi}_{L^2}\) \cite{MazyaNagel}.
In the limit \(\alpha=0\), \(I_\alpha\) is Dirac's delta and theorem\thinspace C yields the fractional Hardy inequality,
\begin{equation}\label{in-Hardy}
\int_{\R^N}\frac{\abs{\varphi (x)}^2}{\abs{x}^s}\dif x
\le \mathcal{C}_{N, 0, s}
\Norm{\varphi}_{\dot{H}^\frac{s}{2}},
\end{equation}
obtained by I.\thinspace Herbst \cite{He1977} and independently by D.\thinspace Yafaev \cite{Yafaev1999}.

The quantitative version of theorem C is

\begin{theoremCp}
Let \(N \ge 1\), \(\alpha \in (0, N)\), \(s\in(0,2)\) and \(s<N\). Then for all \(\varphi \in \dot{H}^{s/2}(\R^N)\), it holds
\begin{multline*}
\mathcal{D}_{N, s}\biggl(\int_{\R^N}\int_{\R^N}\frac{\bigabs{\varphi(x)-\varphi(y)}^2}{\abs{x-y}^{N+s}}\dif x  \dif y - \int_{\R^N} \int_{\R^N}
\frac{\bigabs{\abs{x}^\frac{N - s}{2}\varphi (x) - \abs{y}^\frac{N - s}{2}\varphi (y) }^2}{\abs{x}^\frac{N - s}{2}\abs{x-y}^{N+s}\abs{y}^\frac{N - s}{2}}\dif x \dif y \biggr)\\*
\shoveleft{\qquad\qquad= \frac{1}{\mathcal{C}_{N, \alpha, s}} \biggl(\int_{\R^N} \int_{\R^N} \frac{\varphi (x)}{\abs{x}^\frac{\alpha+s}{2}}
I_\alpha (x - y) \frac{\varphi (y)}{\abs{y}^\frac{\alpha+s}{2}}\dif x \dif y} \\*
+ \frac{1}{2} \int_{\R^N} \int_{\R^N}
\frac{I_\alpha (x - y)}{\abs{x}^\frac{N + \alpha}{2} \abs{y}^\frac{N + \alpha}{2}}
\bigabs{\varphi (x)\abs{x}^\frac{N-s}{2} - \varphi (y) \abs{y}^\frac{N-s}{2}}^2\dif x \dif y\biggr).
\end{multline*}
\end{theoremCp}

The groundstate representation of theorem\thinspace C\cprime{} implies that the infimum in theorem\thinspace C is never achieved in \(\dot{H}^{s/2}(\R^N)\). In fact, the form of the remainder terms suggests that optimality in theorem\thinspace C is related to functions \(\varphi\) that satisfy  \(\varphi (x) \approx \abs{x}^{- (N-s) / 2}\) for \(x \in \R^N \setminus \{0\}\).

In the limiting case \(\alpha=0\) the nonlocal remainder term was obtained by R.\thinspace Frank, E.\thinspace Lieb and R.\thinspace Seiringer \cite{FLS2008}*{section 4} (see also \cite{FrSe2008} and \cite{Be2012}) in a nonlocal groundstate representation for the fractional Hardy inequality,
\begin{multline*}
\mathcal{D}_{N, s}\biggl(\int_{\R^N}\int_{\R^N}\frac{\bigabs{\varphi(x)-\varphi(y)}^2}{\abs{x-y}^{N+s}}\dif x  \dif y - \int_{\R^N} \int_{\R^N}
\frac{\bigabs{\abs{x}^\frac{N - s}{2}\varphi (x) - \abs{y}^\frac{N - s}{2}\varphi (y)}^2}{\abs{x}^\frac{N - s}{2}\abs{x-y}^{N+s}\abs{y}^\frac{N - s}{2}}\dif x \dif y \biggr)\\
= \frac{1}{\mathcal{C}_{N, 0, s}} \int_{\R^N} \frac{\abs{\varphi (x)}^2}{\abs{x}^2}\dif x .
\end{multline*}

A nonlocal groundstate transformation for a local Schr\"odinger operator is derived in section~\ref{sectionLocalGroundstate}, from which theorems A, A\cprime, B and B\cprime{} are deduced in section~\ref{sectionLocalInequalities}.
A general version of a groundstate representations for fractional operators is then obtained in section~\ref{sectionNonLocalGroundstate} below.
The proof of theorem C and C\cprime{} is given in Section \ref{sectionNonLocalInequalities}.

\section{A nonlocal groundstate representation for a Schr\"odinger operator}
\label{sectionLocalGroundstate}

Recall that if \(u>0\) is a solution of the local Schr\"odinger equation
\begin{equation}\label{local-Schrodinger}
-\Delta u+Vu=0\qquad \text{ in \(\R^N\)},
\end{equation}
then for all \(\varphi\in C^\infty_c(\R^N)\) it holds
\begin{equation}\label{AAP}
\int_{\R^N}\abs{\nabla \varphi^2}+\int_{\R^N} V\varphi^2=\int_{\R^N}\Bigabs{\nabla \Bigl( \frac{\varphi}{u} \Bigr)}^2 u^2.
\end{equation}
This identity can be derived by simply testing the equation \eqref{local-Schrodinger} against \(\frac{\varphi^2}{u}\) and by integrating by parts.

Identity \eqref{AAP} is known as the {\em groundstate representation} of the Schr\"odinger operator \(-\Delta+V\) with respect to a positive solution \(u\).
It was discovered independently by W.\thinspace Allegretto \cite{Allegretto-74} and J.\thinspace Piepenbrink \cite{Piepenbrink-74}.
We refer the readers to the lecture notes \cites{Ag1983,Ag1985} by S.\thinspace Agmon for a review of powerful applications
of the groundstate representation in the context of general second order elliptic operators on Riemannian manifolds.

In this section we are going to derive a version of the groundstate representation
for the nonlocal equation with a Schr\"odinger operator and an additional integral operator in the right-hand side,
\begin{equation*}
-\Delta u + Vu=\int_{\R^N} K (\cdot, y) u (y) \dif y\qquad \text{in \(\R^N\)}.
\end{equation*}
Nonlocal linear equation with such structure occur, for instance,
in the analysis of non\-linear Choquard (Schr\"odinger--Newton) equations, where groundstate representations
become an important tool for proving decay bounds on the solutions and nonlinear Liouville theorems \cite{MoVS2012}.

\begin{proposition}
\label{propositionGroundStateloc}
Let \(\Omega \subset \R^N\), \(u \in H^1_\mathrm{loc} (\Omega)\), \(A \in L^\infty_{\mathrm{loc}} (\Omega; \Lin(\R^N; \R^N))\) be self-adjoint almost everywhere in \(\Omega\), \(V \in L^1_\mathrm{loc} (\Omega)\) and \(K : \Omega \times \Omega \to [0, \infty)\) measurable such that for every \(x, y \in \Omega\), \(K (x, y) = K (y, x)\).
If \(V u \in L^1_{\mathrm{loc}} (\Omega)\), \(u^{- 1} \in L^\infty_{\mathrm{loc}} (\Omega)\) and for every nonnegative \(\psi \in C^1_c (\Omega)\),
\[
  \int_{\Omega} A[\nabla u] \cdot \nabla \psi + V u \psi = \int_{\Omega} \int_{\Omega} K (x, y) u (x) \psi (y) \dif y\dif x ,
\]
then for every \(\varphi \in C^1_c (\Omega)\),
\begin{multline*}
\int_{\Omega} A[\nabla \varphi] \cdot \nabla \varphi + V \abs{\varphi}^2
 = \int_{\Omega} \int_{\Omega} K (x, y) \,\varphi (x) \,\varphi (y) \dif y\dif x  + \int_{\Omega} A \Bigl[\nabla \Bigl(\frac{\varphi}{u}\Bigr)\Bigr] \cdot \nabla \Bigl(\frac{\varphi}{u}\Bigr)\\
 + \frac{1}{2} \int_{\Omega} \int_{\Omega} K (x,  y) \,u (x) \,u (y)\, \Bigabs{\frac{\varphi (x)}{u (x)} - \frac{\varphi (y)}{u (y)} }^2\dif y\dif x .
\end{multline*}
\end{proposition}

In the local case \(s=2\) an adaptation of groundstate representation \eqref{AAP}
to distributional solutions \(u\in L^1_{\loc}(\R^N)\) and singular potentials \(V\in L^1_{\loc}(\R^N)\)
was developed in \cite{Fall-II}*{lemma 1.4} (see also \citelist{\cite{CyconFroeseKirschSimon}*{theorem 2.12}\cite{Fall-I}*{lemma B.1}\cite{MoVS2012}*{proposition 3.1}}).

\begin{proof}[Proof of proposition~\ref{propositionGroundStateloc}]
First note that since \(A\) is locally bounded and \(V u \in L^1 (\Omega)\), by classical regularization arguments, we can take nonnegative and compactly supported \(\psi \in H^1 (\Omega) \cap L^\infty (\Omega)\) as test functions.
In particular, we can thus take \(\psi = \varphi^2 / u\) as a test function.
We compute, since \(A (x)\) is self-adjoint for almost every \(x \in \Omega\), that
\[
  A [\nabla \varphi] \cdot \nabla \varphi =
A[\nabla u] \cdot \nabla \Bigl( \frac{\varphi^2}{u}\Bigr)
+ A \Bigl[\nabla \Bigl(\frac{\varphi}{u}\Bigr)\Bigr] \cdot \nabla \Bigl(\frac{\varphi}{u}\Bigr)
\]
and since \(K\) is symmetric,
\[
\begin{split}
 \int_{\Omega}\int_{\Omega} K (x, y)\,u (x)\, \frac{\varphi (y)^2}{u (y)} \dif y\dif x
= &\frac{1}{2} \int_{\Omega}\int_{\Omega} K (x, y)\, \Bigl( u (x) \frac{\varphi (y)^2}{u (y)} + u (y) \frac{\varphi (x)^2}{u (x)}\Bigr) \dif y\dif x \\
= & \int_{\Omega}\int_{\Omega} K (x, y)\, \varphi (x)\, \varphi (y) \dif y\dif x \\
 & \qquad+ \frac{1}{2} \int_{\Omega} \int_{\Omega} K (x,  y)\, u (x)\, u (y)\, \Bigl(\frac{\varphi (x)}{u (x)} - \frac{\varphi (y)}{u (y)} \Bigr)^2\dif y\dif x ;
\end{split}
\]
this yields the conclusion.
\end{proof}

\section{Proofs of the inequalities of Herbst and Beckner}
\label{sectionLocalInequalities}

We first prove the quantitative version of Herbst's inequality:

\begin{proof}[Proof of theorem\thinspace A\cprime]
Take \(u (x) = \abs{x}^{- \frac{N}{2}}\) and
\[
  K (x, y) = \frac{I_\alpha (x - y)}{\abs{x}^\frac{\alpha}{2}\abs{y}^\frac{\alpha}{2}}.
\]
By the semigroup property of the Riesz potentials \cite{Ri1949}, for every \(y \in \R^N \setminus \{0\}\),
\[
 \int_{\R^N \setminus \{0\}} K (x, y) u (x)\dif x
 = \int_{\R^N \setminus \{0\}} \frac{1}{\abs{x}^\frac{N + \alpha}{2}} I_\alpha (x - y) \frac{1}{\abs{y}^\frac{\alpha}{2}}\dif y
 = \frac{1}{2^\alpha} \Bigl(\frac{\Gamma (\frac{N - \alpha}{4})}{\Gamma (\frac{N + \alpha}{4})}\Bigr)^2  \frac{1}{\abs{y}^{\frac{N}{2}}}.
\]
By proposition~\ref{propositionGroundStateloc} with \(\Omega = \R^N \setminus \{0\}\), \(A = 0\) and \(V = \frac{1}{2^\alpha} \frac{\Gamma (\frac{N - \alpha}{4})^2}{\Gamma (\frac{N + \alpha}{4})^2}\), we have the conclusion for \(\varphi \in C^1_c (\R^N \setminus \{0\})\). If \(\varphi \in L^2 (\R^N)\), one uses a classical density argument, passing to the limit with the help of the inequality.
\end{proof}

We now show how theorem\thinspace A can be deduced from theorem\thinspace A\cprime.

\begin{proof}[Proof of theorem\thinspace A]
Take \(\eta \in C ((0, \infty);[0, 1])\) such that \(\eta = 1\) on \((0, 1)\), \(\eta = 0\) on \((2, \infty)\).
Define for \(\lambda \ge 1\),
\[
  u_\lambda (x) = \eta \Bigl(\frac{\abs{x}}{\lambda}\Bigr) \eta \Bigl(\frac{1}{\lambda\abs{x}}\Bigr) \frac{1}{\abs{x}^{\frac{N}{2}}}.
\]
and estimate
\begin{multline*}
 \int_{\R^N} \int_{\R^N}
\frac{I_\alpha (x - y)}{\abs{x}^\frac{N}{2} \abs{y}^\frac{N}{2}}\,
\bigabs{u_\lambda (x)\abs{x}^\frac{N}{2} - u_\lambda (y) \abs{y}^\frac{N}{2}}^2\dif x \dif y\\
\le \int_{\R^{2 N} \setminus (B_\lambda \setminus B_{1/\lambda})^2 \setminus (B_{1 / 2 \lambda} \cup \R^N \setminus B_{2 \lambda})^2}
\frac{I_\alpha (x - y)}{\abs{x}^\frac{N}{2} \abs{y}^\frac{N}{2}} \dif x \dif y \hphantom{\qquad\qquad}\\
\le 2 \int_{B_{2 \lambda}} \int_{\R^N \setminus B_\lambda} \frac{I_\alpha (x - y)}{\abs{x}^\frac{N}{2} \abs{y}^\frac{N}{2}} \dif x \dif y 
+ 2 \int_{B_{1/ \lambda}} \int_{\R^N \setminus B_{1 / 2 \lambda}} \frac{I_\alpha (x - y)}{\abs{x}^\frac{N}{2} \abs{y}^\frac{N}{2}} \dif x \dif y.
\end{multline*}
By scale invariance, it suffices to note that
\[
 \int_{B_{2}} \int_{\R^N \setminus B_1} \frac{1}{\abs{x}^\frac{N}{2} \abs{x - y}^{N - \alpha} \abs{y}^\frac{N}{2}} \dif x \dif y < \infty,
\]
to show that
\[
 \sup_{\lambda \ge 1} \int_{\R^N} \int_{\R^N}
\frac{I_\alpha (x - y)}{\abs{x}^\frac{N}{2} \abs{y}^\frac{N}{2}}\,
\bigabs{u_\lambda (x)\abs{x}^\frac{N}{2} - u_\lambda (y) \abs{y}^\frac{N}{2}}^2\dif x \dif y < \infty.
\]
Since
\[
 \lim_{\lambda \to \infty} \int_{\R^N} \abs{u_\lambda}^2 = \infty,
\]
the conclusion follows.
\end{proof}

Now we consider Beckner's inequality:

\begin{proof}[Proof of theorem\thinspace B\/\cprime]
We begin as in the proof of theorem\thinspace A\cprime, taking for every \(x, y \in \R^N \setminus \{0\}\), \(u (x) = \abs{x}^{- \frac{N - 2}{2}}\) and
\[
  K (x, y) = \frac{I_\alpha (x - y)}{\abs{x}^\frac{\alpha + 2}{2}\abs{y}^\frac{\alpha + 2}{2}}.
\]
We compute now, by the semigroup property of the Riesz potentials, for every \(y \in \R^N \setminus \{0\}\),
\[
 \int_{\R^N \setminus \{0\}} K (x, y)\, u (x)\dif x
= \frac{1}{2^\alpha} \Bigl(\frac{\Gamma (\frac{N - \alpha}{4})}{\Gamma (\frac{N + \alpha}{4})}\Bigr)^2  \frac{1}{\abs{y}^{\frac{N + 2}{2}}}.
\]
On the other hand, we have for every \(x \in \R^N \setminus \{0\}\),
\[
  - \Delta u (x) =  \Bigl( \frac{N - 2}{2} \Bigr)^2 \frac{1}{\abs{x}^\frac{N + 2}{2}},
\]
the conclusion follows from proposition~\ref{propositionGroundStateloc}, taking \(A (x) = \frac{1}{2^{\alpha - 2}} \Bigl(\frac{\Gamma (\frac{N - \alpha}{4})}{(N - 2)\Gamma (\frac{N + \alpha}{4})}\Bigr)^2\) and \(V = 0\) and a density argument.
\end{proof}

We finally show how theorem\thinspace B follows:

\begin{proof}[Proof of theorem\thinspace B]
Choose \(\eta\) as in the proof of theorem\thinspace A, and define now
\[
  u_\lambda (x) = \eta \Bigl(\frac{\abs{x}}{\lambda}\Bigr) \eta \Bigl(\frac{1}{\lambda\abs{x}}\Bigr) \frac{1}{\abs{x}^{\frac{N - 2}{2}}}.
\]
One has, as in the proof of theorem\thinspace A,
\[
\sup_{\lambda \ge 1}
\int_{\R^N} \int_{\R^N}
\frac{I_\alpha (x - y)}{\abs{x}^\frac{N + \alpha}{2} \abs{y}^\frac{N + \alpha}{2}}
\bigl(u_\lambda (x)\abs{x}^\frac{N - 2}{2} - u_\lambda (y) \abs{y}^\frac{N - 2}{2}\bigr)^2\dif x \dif y < \infty
\]
and
\[
 \lim_{\lambda \to \infty} \int_{\R^N} \abs{\nabla u_\lambda}^2 = \infty.
\]
In order to conclude, note that if \(\lambda \ge 1\),
\begin{equation*}
\begin{split}
 \int_{\R^N} \frac{\bigabs{\nabla(\abs{x}^\frac{N - 2}{2} u_\lambda (x))}^2}{\abs{x}^{N - 2}} \dif x
&= \int_{B_{2 \lambda} \setminus B_\lambda} \frac{\eta'(\abs{x}/\lambda)^2}{\lambda^2 \abs{x}^{N - 2}} \dif x +
\int_{B_{1/ \lambda} \setminus B_{1 / 2 \lambda}} \frac{\lambda^2 \eta'(\lambda/\abs{x})^2}{\abs{x}^{N + 2}} \dif x\\
&= \int_{B_{2} \setminus B_1} \frac{\eta'(\abs{z})^2}{\abs{z}^{N - 2}} \dif z +
\int_{B_{1} \setminus B_{1 / 2}} \frac{\eta'(1/\abs{z})^2}{\abs{z}^{N + 2}} \dif z,
\end{split}
\end{equation*}
whose right-hand side does not depend on \(\lambda\).
\end{proof}

\section{A nonlocal groundstate representation for the fractional Laplacian}\label{sectionNonLocalGroundstate}

A version of a nonlocal groundstate representation for the fractional Laplacian \((-\Delta)^{s/2}\) with \(0<s<2\)
was introduced by R.\thinspace Frank, E.\thinspace Lieb and R.\thinspace Seiringer in \cite{FLS2008}*{section 3} and \cite{FrSe2008},
where (amongst other things) it was used to obtain an alternative proof of the fractional Hardy's inequality \eqref{in-Hardy}.

In this section we are going to derive a version of the groundstate representation
for the nonlocal equation with a fractional Laplacian in the left and an integral operator in the right-hand side.

\begin{proposition}
\label{propositionGroundState}
Let \(N \ge 1\), \(s \in (0, 2)\) and \(s<N\).
Let \(K : \R^N \times \R^N \to [0, \infty)\) be measurable symmetric,
that is \(K (x, y) = K (y, x)\) for a.e.~\(x, y \in \R^N\).
Let $u\in L^1_\loc(\R^N)$ and assume that
\begin{equation*}
\int_{\R^N} \frac{u (x)}{1 + \abs{x}^{N + s}} \,d x < \infty\quad\text{and}\quad
\int_{\R^N} K (\cdot, y)\, u (y) \dif y \in L^1_{\mathrm{loc}} (\R^N).
\end{equation*}
If for every nonnegative \(\psi \in C^\infty_c (\R^N)\) it holds
\begin{equation*}
  \int_{\R^N}  \int_{\R^N} \frac{\bigl(u (x) - u (y)\bigr)\bigl(\psi (x) - \psi (y)\bigr)}{\abs{x - y}^{N + s}} \dif x \dif y \,  = \int_{\R^N} \int_{\R^N} K (x, y)\,u (y)\,\psi (x) \dif y\dif x ,
\end{equation*}
and \(u^{- 1} \in L^\infty_{\mathrm{loc}} (\R^N)\), then for every \(\varphi \in \dot{H}^{s/2}(\R^N)\) it holds
\begin{multline*}
\int_{\R^N} \int_{\R^N} \frac{\abs{\varphi (x) - \varphi (y)}^2}{\abs{x - y}^{N + s}} \dif x \dif y
 = \int_{\R^N} \int_{\R^N} K (x, y)\, \varphi (x)\, \varphi (y) \dif y\dif x \\*
 \shoveright{+ \frac{1}{2} \int_{\R^N} \int_{\R^N} K (x,  y)\, u (x)\, u (y) \Bigl(\frac{\varphi (x)}{u (x)} - \frac{\varphi (y)}{u (y)} \Bigr)^2\dif y\dif x } \\*
 + \int_{\R^N}\int_{\R^N}\Bigabs{\frac{\varphi(x)}{u(x)}-\frac{\varphi(y)}{u(y)}}^2 \frac{u(x)\, u(y) }{\abs{x-y}^{N+s}}\dif x  \dif y.
\end{multline*}
\end{proposition}

In the case \(K(x,y)=0\) the above result improves upon \cite{FLS2008}*{proposition 4.1},
where instead of \(\int_{\R^N} \frac{u (x)}{1 + \abs{x}^{N + s}} \,d x< \infty\) a stronger assumption \(u\in H^{s/2}(\R^N)\) was required.
A similar improvement was obtained recently in \cite{Fall-III}*{lemma 2.10}.
In Section \ref{sectionNonLocalInequalities} we will use the groundstate representation with respect to a function \(u(x)=\abs{x}^{- (N-s) / 2}\not\in H^{s/2}_\loc(\R^N)\),
so such improvement is indeed important.

\begin{proof}[Proof of proposition~\ref{propositionGroundState}]

First note that since \(u^{-1}\in L^\infty_{\mathrm{loc}}(\R^N)\), for arbitrary \(\varphi\in C^\infty_c(\R^N)\)
we have \(\psi = \varphi^2 / u\in L^\infty_c(\R^N)\).

Let \(\eta\in C^\infty_c(\R^N)\) be such that \(\supp \eta \subset B_1\), \(\int_{\R^N} \eta = 1\) and \(\eta \ge 0\).
For \(\delta>0\) and \(x \in \R^N\), let \(\eta_\delta(x)=\delta^{-N}\eta (x/\delta)\) and let \(\check{\eta}_{\delta} (x) = \eta_\delta (-x)\).
Given \(\varphi \in C^\infty_c(\Omega)\) and \(\delta>0\), we can thus take \(\psi_\delta=\check{\eta}_\delta \ast \frac{\varphi^2}{\eta_\delta \ast u}\in C^\infty_c(\Omega)\) as a test function in the equation.
We will handle each of the terms separately.

Since \(u \in L^1_\loc(\R^N)\), we have \(\eta_\delta \ast u \to u\) in \(L^1_\mathrm{loc} (\R^N)\) and almost everywhere in \(\R^N\) as \(\delta \to 0\).
By our assumption and Lebesgue's dominated convergence, we obtain
\begin{multline*}
 \int_{\R^N}\int_{\R^N} K (x, y) \,u (x)\, \Bigl(\check{\eta}_\delta \ast \frac{\varphi^2}{\eta_\delta \ast u}\Bigr)(y)\dif y\dif x =
 \int_{\R^N}\Big(\eta_\delta \ast\int_{\R^N} K (x, y)\, u (x)\dif x \Big)\, \frac{\varphi^2}{\eta_\delta \ast u}(y)\dif y\\
 \to \int_{\R^N}\int_{\R^N} K (x, y)\, u (x)\, \frac{\varphi^2}{u}(y)\dif y\dif x ,
\end{multline*}
as \(\delta\to 0\).
Since \(K\) is symmetric, as in the proof of proposition~\ref{propositionGroundStateloc}, the latter could be transformed as
\begin{multline*}
\int_{\R^N}\int_{\R^N} K (x, y)\, u (x) \frac{\varphi (y)^2}{u (y)} \dif y\dif x \\*
=\frac{1}{2} \int_{\R^N}\int_{\R^N} K (x, y) \Bigl( u (x) \frac{\varphi (y)^2}{u (y)} + u (y) \frac{\varphi (x)^2}{u (x)}\Bigr) \dif y\dif x \\
=  \int_{\R^N}\int_{\R^N} K (x, y)\, \varphi (x)\, \varphi (y) \dif y\dif x \\
  + \frac{1}{2} \int_{\R^N} \int_{\R^N} K (x,  y)\, u (x)\, u (y)\, \Bigl(\frac{\varphi (x)}{u (x)} - \frac{\varphi (y)}{u (y)} \Bigr)^2\dif y\dif x .
\end{multline*}

In the left-hand side, by a change of variable
\begin{multline*}
\int_{\R^N}\int_{\R^N}\frac{\bigl(u(x) - u(y)\bigr)\Big(\bigl(\check{\eta_\delta} \ast \frac{\varphi^2}{\eta_\delta \ast u}(x)\bigr)-\bigl(\check{\eta_\delta} \ast \frac{\varphi^2}{\eta_\delta \ast u}\bigr)(y)\Big)}{\abs{x-y}^{N+s}}\dif x  \dif y\\
=\int_{\R^N}\int_{\R^N} \int_{\R^N} \frac{\bigl( u(x)-u(y)\big)\Big(\eta_\delta (-z) \frac{\varphi^2}{\eta_\delta \ast u}(x-z)-\eta_\delta (-z)\frac{\varphi^2}{\eta_\delta \ast u}(y-z)\Big)}{\abs{x-y}^{N+s}}\dif z \dif x  \dif y\\
\shoveright{=\int_{\R^N}\int_{\R^N} \int_{\R^N} \frac{\big(\eta_\delta (w) u(x-w)-\eta_\delta (w) u(y - w)\bigr)\Big(\frac{\varphi^2}{\eta_\delta \ast u}(x)-\frac{\varphi^2}{\eta_\delta \ast u}(y)\Big)}{\abs{x-y}^{N+s}}\dif w \dif x  \dif y}\\
=\int_{\R^N}\int_{\R^N}\frac{\big((\eta_\delta \ast u)(x)-(\eta_\delta \ast u)(y)\big)\Big(\frac{\varphi^2}{\eta_\delta \ast u}(x)-\frac{\varphi^2}{\eta_\delta \ast u}(y)\Big)}{\abs{x-y}^{N+s}}\dif x  \dif y.
\end{multline*}
Similarly to \cite{FLS2008}*{proposition 4.1}, we obtain
\begin{multline*}\label{polar}
\int_{\R^N}\int_{\R^N}\frac{\big((\eta_\delta \ast u)(x)-(\eta_\delta \ast u)(y)\big)\Big(\frac{\varphi^2}{\eta_\delta \ast u}(x)-\frac{\varphi^2}{\eta_\delta \ast u}(y)\Big)}{\abs{x-y}^{N+s}}\dif x  \dif y\\
=\int_{\R^N}\int_{\R^N}\frac{\abs{\varphi(x)-\varphi(y)}^2-\bigabs{\frac{\varphi}{\eta_\delta \ast u}(x)-\frac{\varphi}{\eta_\delta \ast u}(y)}^2 (\eta_\delta \ast u)(x)\, (\eta_\delta \ast u)(y)}{\abs{x-y}^{N+s}}\dif x  \dif y\\
\to\int_{\R^N}\int_{\R^N}\frac{\abs{\varphi(x)-\varphi(y)}^2-\bigabs{\frac{\varphi}{u}(x)-\frac{\varphi}{u}(y)}^2 u(x) u(y)}{\abs{x-y}^{N+s}}\dif x  \dif y,
\end{multline*}
as \(\delta\to 0\), again by Lebesgue's dominated convergence theorem.
\end{proof}

\section{Proofs of the inequalities in fractional spaces}\label{sectionNonLocalInequalities}

In order to deduce the quantitative groundstate representation of theorem\thinspace C\cprime{}
from its more general version of proposition~\ref{propositionGroundState}.

\begin{proof}[Proof of theorem\thinspace C\/\cprime]
In proposition \ref{propositionGroundState}, take
\begin{align*}
u (x) &= \frac{1}{\abs{x}^{\frac{N-s}{2}}} &
&\text{and }&
K (x, y) &= \mathcal{C}_{N, \alpha, s}\frac{I_\alpha (x - y)}{\abs{x}^\frac{\alpha+s}{2}\abs{y}^\frac{\alpha+s}{2}}.
\end{align*}
Since by \cite{MazyaNagel},
\[
 \mathcal{D}_{N, s}\int_{\R^N}\int_{\R^N}\frac{\bigl(\varphi(x)-\varphi(y)\bigr)\bigl(u (x) - u(y)\bigr)}{\abs{x-y}^{N+s}}\dif x  \dif y = \int_{\R^N} \overline{\widehat{\varphi} (\xi)}\, \abs{\xi}^s\, \widehat{u} (\xi)\,d\xi
\]
and
\[
 \Hat{I}_{\gamma} (\xi) = \frac{1}{\abs{\xi}^{\gamma}},
\]
where the Fourier transform \(\hat\varphi\) is defined for every \(\xi \in \R^N\) by
\[
 \hat\varphi(\xi)=\frac{1}{(2 \pi)^\frac{N}{2}}\int_{\R^N}\varphi(x)e^{-i \xi\cdot x}\dif x ,
\]
we compute
\begin{equation*}
\begin{split}
\mathcal{D}_{N, s}\int_{\R^N}\int_{\R^N}\frac{\bigl(\varphi(x)-\varphi(y)\bigr)\bigl(u (x) - u(y)\bigr)}{\abs{x-y}^{N+s}}\dif x  \dif y
&= \int_{\R^N} \overline{\widehat{\varphi} (\xi)}\, \abs{\xi}^s \, I_{\frac{N - s}{2}} (\xi)\,d\xi\\
&= 2^{s}\Bigl(\frac{\Gamma\big(\frac{N+s}{4}\big)}{\Gamma\big(\frac{N-s}{4}\big)}\Bigr)^2 \int_{\R^N} \overline{\widehat{\varphi} (\xi)}\, \abs{\xi}^s\, I_{\frac{N + s}{2}} (\xi)\,d\xi\\
&= 2^{s} \Bigl(\frac{\Gamma\big(\frac{N+s}{4}\big)}{\Gamma\big(\frac{N-s}{4}\big)}\Bigr)^2 \int_{\R^N} \frac{\varphi (x)}{\abs{x}^{\frac{N + s}{2}}}\dif x .
\end{split}
\end{equation*}

On the other hand, by the semigroup property of the Riesz potentials \cite{Ri1949},
for \(0 < \alpha <\beta< N\)
\[
\frac{1}{\abs{x}^{\frac{\alpha+s}{2}}}\bigg(I_\alpha\ast\frac{1}{\abs{x}^{-\frac{N+\alpha}{2}}}\bigg)\\
=
2^{-\alpha}\Bigl(\frac{\Gamma\bigl(\tfrac{N-\alpha}{4}\bigr)}{\Gamma\big(\tfrac{N+\alpha}{4}\big)}
\Bigr)^2 \frac{1}{\abs{x}^{\frac{N+s}{2}}}.
\]
By proposition~\ref{propositionGroundState}, we reach the required conclusion for \(\varphi \in C^\infty_c(\R^N)\).
If \(\varphi \in H^{s/2}(\R^N)\), one uses a classical density argument,
passing to the limit with the help of inequality \eqref{in-SW}.
\end{proof}

We now show how optimality of the constant \(\mathcal{C}_{N, \alpha, s}\) can be deduced using the remainder terms of the groundstate representation of theorem\thinspace C\cprime.

\begin{proof}[Proof of theorem\thinspace C from theorem\thinspace C\/\cprime]
Take \(\eta \in C ((0, \infty);[0, 1])\) such that \(\eta = 1\) on \((0, 1)\), \(\eta = 0\) on \((2, \infty)\).
Define for \(s\in(0, 2)\) and \(\lambda \ge 1\),
\[
  u_\lambda(x): = \eta \Bigl(\frac{\abs{x}}{\lambda}\Bigr) \eta \Bigl(\frac{1}{\lambda\abs{x}}\Bigr) \frac{1}{\abs{x}^{\frac{N-s}{2}}}.
\]
We shall estimate the remainders in theorem\thinspace C\cprime.

For \(\alpha\in(0,N)\) we obtain
\begin{equation*}
\begin{split}
\mathcal{J}_\alpha(u_\lambda)&:=\int_{\R^N} \int_{\R^N}
\frac{I_\alpha (x - y)}{\abs{x}^\frac{N+\alpha}{2} \abs{y}^\frac{N+\alpha}{2}}
\bigabs{u_\lambda (x)\abs{x}^\frac{N-s}{2} - u_\lambda (y) \abs{y}^\frac{N-s}{2}}^2\dif x \dif y\\
&\le 2 \int_{B_{2 \lambda}} \int_{\R^N \setminus B_\lambda} \frac{I_\alpha (x - y)}{\abs{x}^\frac{N+\alpha}{2} \abs{y}^\frac{N+\alpha}{2}} \dif x \dif y
+ 2 \int_{B_{1/ \lambda}} \int_{\R^N \setminus B_{1 / 2 \lambda}} \frac{I_\alpha (x - y)}{\abs{x}^\frac{N+\alpha}{2} \abs{y}^\frac{N+\alpha}{2}} \dif x \dif y.
\end{split}
\end{equation*}
By scale invariance, it suffices to note that
\[
 \int_{B_{2}} \int_{\R^N \setminus B_1} \frac{I_\alpha(x-y)}{\abs{x}^\frac{N+\alpha}{2}\abs{y}^\frac{N+\alpha}{2}} \dif x \dif y < \infty,
\]
in order to conclude that
\[
 \limsup_{\lambda \to \infty} \mathcal{J}_\alpha(u_\lambda) < \infty.
\]

For \(0<s<\min\{2,N\}\) we obtain
\begin{multline*}
\Rem_s(u_\lambda):=
\int_{\R^N} \int_{\R^N}
\frac{\bigabs{u_\lambda(x)\abs{x}^\frac{N - s}{2}- u_\lambda(y)\abs{y}^\frac{N - s}{2}}^2}{\abs{x}^\frac{N - s}{2}\abs{x-y}^{N+s}\abs{y}^\frac{N - s}{2}}\dif x \dif y,\\
\le \int_{B_{2 \lambda}} \int_{\R^N \setminus B_\lambda} \frac{1}{\abs{x}^\frac{N-s}{2}\abs{x-y}^{N+s}\abs{y}^\frac{N-s}{2}} \dif x \dif y\\
+ \int_{B_{1/ \lambda}} \int_{\R^N \setminus B_{1 / 2 \lambda}} \frac{1}{\abs{x}^\frac{N-s}{2}\abs{x-y}^{N+s}\abs{y}^\frac{N-s}{2}} \dif x \dif y.
\end{multline*}
As before, note that
\[
 \int_{B_{2}} \int_{\R^N \setminus B_1} \frac{1}{\abs{x}^\frac{N-s}{2}\abs{x-y}^{N+s}\abs{y}^\frac{N-s}{2}} \dif x \dif y < \infty,
\]
in order to conclude by scale invariance that for \(s \in (0, 2)\),
\[
 \limsup_{\lambda \to \infty} \Rem_s(u_\lambda) < \infty.
\]

Finally, note that
\begin{multline*}
  \lim_{\lambda \to \infty} \int_{\R^N}\int_{\R^N}
\frac{\bigabs{u_\lambda (x)-u_\lambda(y)}^2}{\abs{x-y}^{N+s}}\dif x \, \dif y
  = \int_{\R^N}\int_{\R^N} \frac{1}{\abs{x-y}^{N+s}} \biggabs{\frac{1}{\abs{x}^\frac{N - s}{2}} - \frac{1}{\abs{y}^\frac{N - s}{2}} }^2\dif x  \dif y = \infty,
\end{multline*}
so the conclusion follows.
\end{proof}

\end{document}